\documentclass[a4paper,11pt]{amsart}  
\usepackage[foot]{amsaddr}

\usepackage{multirow}
\usepackage{bm}
\usepackage{amsfonts}
\usepackage{amsmath}
\usepackage{amssymb}
\usepackage{mathrsfs}
\usepackage{graphicx}
\usepackage{amsthm}
\usepackage{commath} 

\usepackage{soul}
\usepackage{lpic}
\usepackage{subfig}

\usepackage{wasysym}
\usepackage{fancybox}
\usepackage{lettrine}

\usepackage{tikz}

\usetikzlibrary{calc}
\usetikzlibrary{matrix}
\usetikzlibrary{positioning}
\usetikzlibrary{shapes.geometric, arrows}
\usetikzlibrary{arrows, decorations.markings}

\usepackage{pgfplotstable}
\usetikzlibrary{spy,calc}

\newcommand{\RadonG}{\Radon_{\mathsf s}}

\newcommand{\RadonGadj}{\Radon_{\mathsf s}^*}
\newcommand{\RadonSadj}{\Radon_{\ell}^*}
\newcommand{\RadonS}{\Radon_{\ell}}  
\newcommand{\BackG}{\Back_{\mathsf s}}
\newcommand{\BackS}{\Back_{\ell}}
\newcommand{\MG}{\mathcal M_{\mathsf s}}
\newcommand{\MS}{\mathcal M_{\ell}}
\newcommand{\JG}{\mathsf J_{\mathsf s}}
\newcommand{\JS}{\mathsf J_{\ell}}
\newcommand{\MGadj}{\mathcal M_{\mathsf s}^*}
\newcommand{\VG}{V_{\mathsf s}}
\newcommand{\VS}{V_{\ell}}
\newcommand{\R}{\mathbb{R}}
\newcommand{\quotient}{\R / 2\pi\mathbb{Z}}
\newcommand{\PIquotient}{\R / \pi\mathbb{Z}}

\newcommand{\Radon}{\mathcal{R}}

\newcommand{\Back}{\mathcal{B}}

\newcommand{\dd}{\mathrm{d}}

\newtheorem{theo}{Theorem}

\theoremstyle{definition}
\newtheorem{remark}{Remark}
\newtheorem{lemma}{Lemma}

\title[Alternative Fan-beam Backprojection and Adjoint Operators]{Alternative Fan-beam Backprojection \\ and Adjoint Operators} 

\author{Patricio Guerrero$^*$}
\address{$^*$Department of Mechanical Engineering, KU Leuven, Leuven, Belgium}
\thanks{Corresponding author: \texttt{patricio.guerrero@kuleuven.be}}
\author{Matheus Bernardi$^\dag$}
\author{Eduardo Miqueles$^\dag$}
\address{$^\dag$Brazilian Synchrotron Light Laboratory, National Center for Research in Energy and Materials, Campinas, Brazil}

\makeatletter
\def\paragraph{\@startsection{paragraph}{4}%
  \z@\z@{-\fontdimen2\font}%
  {\normalfont\itshape}}
\makeatother

\usepackage{hyperref}

\begin{document}


\maketitle
\begin{abstract}
We present in this work alternative analytic formulations for the fan-beam tomographic backprojection operation and its associated adjoint transform in standard (equiangular) and linear (equidistant) detector geometries. The proposed formulations are obtained from a recent backprojection theorem in parallel tomography. Such formulations are written as a Bessel-Neumann series in the frequency domain that can be implemented as an $O(N^{2.3729})$ matrix multiplication. Proofs are provided together with numerical simulations compared with conventional fan-beam $O(N^{3})$ backprojection representations showing more robustness when dealing with highly noisy data.  

\vspace{2em}


\noindent \textbf{Published in:} \href{https://doi.org/10.1515/jiip-2022-0029}{\emph{J. of Inverse and Ill-Posed Problems}, 2023.} 
\end{abstract}

\section{Introduction}

Fan-beam tomographic measurements are used in different modalities of non-destructive imaging, as those obtained using X-rays. A typical tomographic device using fan-beam geometry is shown in Figure \ref{fig:setup}. 
This is a widely used and  known technique, and there are many reconstruction algorithms for this configuration. After being generated with a given aperture angle and a fixed source-detector distance, the wavefront hits the sample originating a signal on the detector. Different propagation regimes can be considered with a varying distance \cite{luke2002optical, miqueles2020}, although we will consider a pure mathematical signal idealized as the fan-beam Radon transform of the given object.
\begin{figure}[t]
\centering
\includegraphics[scale=0.7]{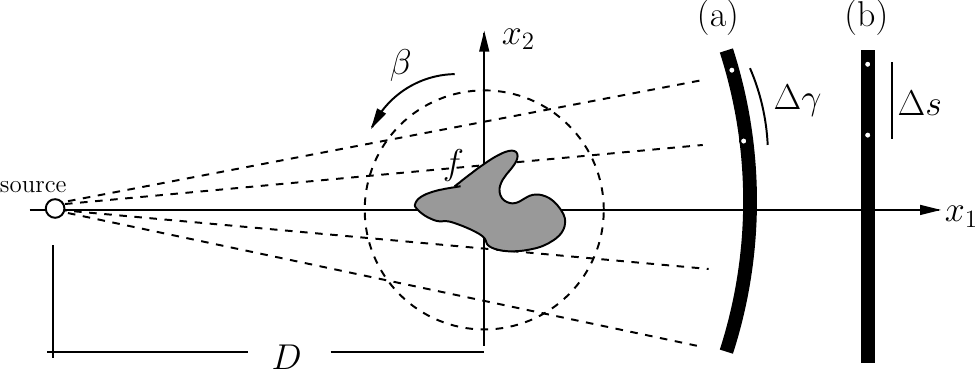}
\caption{Tomographic setup for fan-beam geometry: (a) \emph{standard} detector, i.e., equispaced angular mesh with $\Delta \gamma$; sinograms obtained here are denoted by $w(\gamma,\beta)$. (b) \emph{linear} detector, i.e., equispaced mesh with $\Delta s$ perpendicular to the central ray; sinograms obtained here are denoted by $g(s,\beta)$.}
\label{fig:setup}
\end{figure}

For the tomographic parallel case by means of the classical backprojection formulation for image reconstruction, we can use the recent backprojection slice theorem formulation \cite{miqueles2018backprojection}. It is a formula that reduces a complete backprojection of a 2D sinogram from a computational cost of $O(N^3)$ to $O(N^2 \log N)$ with $N\times N$ pixels in the final reconstructed image and $N$ tomographic rotations. In this work, we take advantage of that formulation for two other popular fan-beam geometries, that are a) equispaced angles within the fan with uniform size $\Delta\gamma$, and b) equispaced points in a linear detector with a mesh step $\Delta s$. As indicated in \cite{natterer2001mathematical}, we refer to each acquisition as \emph{standard fan} and \emph{linear fan}, respectively, and these are illustrated in Figure \ref{fig:setup}.(a) and \ref{fig:setup}.(b). The linear case is easier to be implemented at a synchrotron beamline with an appropriate optical setup, whereas the former is more used with medical or industrial X-rays scans. 
For completeness, we discuss both standard and linear cases.

\paragraph*{Related work, contributions and overview}

Further work has been done after the conventional $O(N^3)$ backprojection strategies for fan-beam sinograms were established in \cite{kak1988}, where both the fan-beam filtered backprojection (FBP) algorithm and a fan-to-parallel rebinning algorithm is formalized. In \cite{zhao2014fan}  a generalized Fourier slice theorem (FST) for fan-beam geometries is introduced presenting the same $O(N^3)$ computational complexity as a conventional backprojection in the parallel case. In addition it requires a rebinning operation on the measured data from fan to parallel geometry and then it generalize the FST. Further relations on the frequency domain are obtained in \cite{kazantsev2008fan} where a fan-to-parallel rebinning is also necessary and the Gerchberg-Papoulis iterative algorithm is used to obtain the final reconstruction; the method is mentioned to be computationally expensive although the computational complexity is not explicitly described. A different approach is established in \cite{chen2005novel} without direct rebinning in the image domain but in the frequency domain of the reconstruction; however, a weighted backprojection of the data is needed and no algorithm with lower cost than $O(N^3)$ is provided to this end. In \cite{wei2005relation}, a series formulation for the fan-beam FBP algorithm is presented where the backprojection is written as the first order approximation of the FBP series. Using a hierarchical approach, \cite{george2013fast} presents an $O(N^2\log N)$ algorithm based on efficient sampling of the fan-beam geometry and without rebinning to parallel projections, although a series of intermediate interpolations are needed on every level of the hierarchical structure. 

Fast algorithms are either based on deriving a mathematical expression relying on low complexity auxiliary transformations, e.g., the fast Fourier transform, or hierarchically, taking a divide-and-conquer approach \cite{bresler2007}. While the method in \cite{george2013fast} relies on the latter approach, this work relies on the former. We present here an analytic formulation for the fan-beam backprojection operation as well as for the associated adjoint transform, considering both standard (Theorem \ref{teostandard}) and linear (Theorem \ref{teolinear}) geometries. Such formulations can be implemented numerically as a multiplication of two matrices, allowing the use of fast matrix multiplication algorithms \cite{matrix, matrix2} resulting in a computational complexity of $O(N^{2.3729})$. It is worth to recall then that our method is computationally more expensive than the $O(N^2\log N)$ hierarchical method of \cite{george2013fast} and the purpose of this work is an alternative analytic and simple formulation of the fan-beam backprojection and adjoint transforms still less expensive than the usual $O(N^3)$ formulation. 

This work is organized as follows. Section 2 presents a review of conventional fan-beam backprojection operators for both standard and linear geometries and the relationship with their associated adjoint transforms, which are defined as bounded operators on Hilbert spaces. In section 3 we rewrite the rebinning strategy in terms of the inverse rebinning operator as a two-step backprojection formula. In this section, the adjoint transforms of both geometries are also formalized as a composition of two operators and therefore, we show that this formulation is equivalent to the one in section 2 from the uniqueness of the adjoint of bounded operators. Section 4 presents our fan-beam adjoint and backprojection theorems as a Bessel-Neumann series representation (referred here as BN series method) and their proofs,
while section 5 numerically supports our method with simulated data where the backprojection is computed as a straightforward matrix multiplication based on the BN series method and its performance is compared with the conventional rebinning strategy.  

\section{Fan-beam backprojections}

\subsection{The 2D Radon transform}

The two-dimensional parallel Radon transform is defined \cite{deans2007radon} as the linear operator $\Radon \colon U \to V$ where $U$ is the Schwartz space $\mathcal S(\R^2)$ of rapidly decreasing functions defined on $\R^2$, so-called \textit{feature} space; and $V$, the \textit{sinogram} space, is the Schwartz space defined on the unit cylinder $\mathcal S(\R \times \quotient)$ where $\quotient$ is the quotient group of reals modulo $2\pi$. 


Let $(t, \theta) \in \R\times \quotient$ and $f \in U$. $\Radon f$ is defined as the integral operator
\begin{equation}\label{radon}
\Radon f(t,\theta) 
= \int_{\R^2} f( \bm x ) \delta (t - \bm x \cdot \bm \xi_\theta) \, \dd \bm x,
\end{equation}
where $\delta$ is the Delta distribution and $\bm\xi_\theta = (\cos\theta, \sin\theta)^T$. 
We will denote by $f_\mathsf{P} \in U$ a feature function acting on $\bm x = (\rho, \phi) \in \R_+ \times [0, 2\pi[ $ expressed in polar coordinates. $\Radon f_\mathsf{P}$ is then written as 
\begin{equation}\label{polar}
\Radon f_\mathsf{P}(t,\theta) 
= \int_0^{2\pi} \int_{\R_+} f_\mathsf{P}( \rho, \phi ) \delta (t - \rho\cos(\theta - \phi)) \, \rho \dd \rho \dd \phi.
\end{equation}
Note that $\Radon f$ is an even function on the unit cylinder, that is, \begin{equation}\label{Rsimmetry}
\Radon f(t, \theta) = \Radon f(-t, \theta+\pi). 
\end{equation}

\paragraph*{Adjoint and backprojection.} Let $T \colon X \to Y$ be a bounded linear operator between Hilbert spaces $X$ and $Y$. Its adjoint operator is defined as the operator $T^*\colon Y \to X$ that verifies 
\begin{equation}\label{adjoint}
	\forall u\in X, \quad \forall v\in Y, \quad \langle Tu , v\rangle_Y = \langle u,  T^* v \rangle_X.
\end{equation}  

The adjoint and backprojection operators of the Radon transform  are closely related as follows \cite{deans2007radon}. Expanding the definitions of $U$ and $V$ to $L^2$ spaces \cite{euclidean}, the adjoint of $\Radon$ is the operator $\Radon^* \colon V \to U$ defined for $p \in V$ and $\bm x \in \R^2$ by
\begin{equation*}\label{backparallel}
    \Radon^* p (\bm x) = \int_0^{2\pi} p(\bm x \cdot \bm \xi_\theta, \theta) \, \mathrm d\theta.
\end{equation*}

With standard $L^2$ inner products on $U$ and $V$, it is clear that $\Radon^*$ verifies (\ref{adjoint}). 

In the FBP algorithm \cite{natterer2001mathematical}, the backprojection operation is applied after filtering the data with an adequate convolutional filter. Such backprojection operator, denoted $\Back$ is defined by limiting the last integral to $[0, \pi]$, and from the symmetry property (\ref{Rsimmetry}) we have $\Radon^* \equiv 2\Back$.

Computing $b=\Back p$ could be extremely expensive for discrete versions of the sinogram $p$, when $\bm x$ covers a domain with a large number of points (pixels in practice) and also $(t, \theta)$ covers a large number of pixels and a variety of angles (according to Crowther's criterion \cite{crowther1970reconstruction}). This is the case for synchrotron tomographic projections using high-resolution detectors with more than $2048\times 2048$ pixels and more than $2048$ angles \cite{gio}. Recently \cite{miqueles2018backprojection}, a low-complexity formulation for computing $b$ was obtained in the frequency domain using polar coordinates, the Backprojection Slice Theorem (BST). Such theorem reads for $p \in V$ and $\sigma > 0$,
\begin{equation}\label{bst}
    \widehat{ \Back p}(\sigma \bm\xi_\theta) = \frac{\hat{p}(\sigma,\theta)}{\sigma},
\end{equation}
where $\widehat{ \Back p}$ is the 2D Fourier transform of $\Back p$ in polar frequency coordinates and $\hat{p}$ is the 1D Fourier transform of $p$ with respect to $t$. The action of $\{\Radon,\Back\}$ is presented in Figure \ref{fig:idea}. It is a well known fact that 
$\{p,f\}$ are related through the FST \cite{natterer2001mathematical, deans2007radon}, 
while $\{p,b\}$ through the BST \cite{miqueles2018backprojection}.
\begin{figure}
\centering
\begin{tikzpicture}
    \filldraw [fill=white, draw=black!60](-1.5,0) circle (1cm);
    \filldraw [fill=white, draw=black!60] (1.5,0) circle (1cm);
    \filldraw [fill=white, draw=black!60] (5.0,0) circle (1cm);
    \node (U) at (-2.6, 0.4) {$U$};
    \node (V) at ( 2.6, 0.4) {$V$};
    \node (V) at ( 5.7, 0.4) {$\VG$};
    
    \node (f) at (-1.6, 0.4) {${f}$};
    \node (b) at (-1.6,-0.5) {${b}$};
    \node (p) at ( 1.6, 0.0) {${p}$}; 
    \node (g) at ( 4.9, 0.0) {${w}$}; 

    \draw[->] (f) -- node[above] {~~$\Radon$} (p);
    \draw[->] (p) -- node[above] {~~$\MG$} (g);
    \draw[->] (p) -- node[below] {~~$\Back$} (b);
    \draw[->] (g) [out=-110, in=-40] to node[above]{$\BackG$} (b);  
\end{tikzpicture}
\caption{Action diagram for operators $\{\Back,\Radon\}$ and $\{\BackG,\RadonG\}$, with $\RadonG = \MG\Radon$, $\BackG = \Back \MG^*$ and
a generalized change of variables $\MG$.}
\label{fig:idea}
\end{figure}


\subsection{Fan-beam geometries}

There exist two main fan-beam parameterizations for the Radon transform \cite{kak1988}, the first $\RadonG \colon U \to \VG$ is referred as the \textit{standard fan-beam} transform, where the domain of sinograms $\VG$ lies within the set $\PIquotient \times \quotient$. The second case, $\RadonS \colon U \to \VS$ is referred as the \textit{linear fan-beam} transform, with sinograms domain $\VS$ varying within the set $\R\times\quotient$. 

Let us consider $(\gamma, \beta) \in \PIquotient\times\quotient$ a pair of angles and $s \in \R$. Acting on a feature function $f\in U$, the linear operators $\RadonG$  and $\RadonS$ are defined respectively as 
\begin{equation*}
w(\gamma,\beta) = \Radon f( D\sin\gamma, \beta + \gamma ),
\end{equation*}
and
\begin{equation*}
g(s,\beta) = \Radon f( \dfrac{sD}{\sqrt{s^2+D^2}}, \beta + \arctan\dfrac{s}{D} ),
\end{equation*}
where $D>0$ is the distance of the focal point source to the center of rotation located at the origin, see Figure \ref{fig:setup}. Hence $w \in \VG$ and $g \in \VS$. The following symmetry relationships hold:  
\begin{equation}\label{symmetry}
    \begin{split}
    w(\gamma, \beta) & = w(-\gamma, \beta + 2\gamma + \pi ),\\
    g(s, \beta) & = g(-s,\beta + 2\arctan\dfrac{s}{D} + \pi ).
    \end{split}
\end{equation}

Since both operations represent a change of variables in the classical parallel sinogram, we can use the following notation:
\begin{equation*}
    \RadonG f(\gamma, \beta) = \MG \Radon f(\gamma,\beta),  \quad 
    \RadonS f(s, \beta) = \MS \Radon f(s,\beta),
\end{equation*}
where $\MG \colon V \to \VG$ and $\MS \colon V \to \VS$ are rebinning operators acting on a parallel sinogram $p$, respectively as
\begin{equation}\label{rebinning}
\begin{array}{ll}
\MG p(\gamma,\beta) = p(D\sin \gamma, \beta + \gamma),\\
\MS p(s,\beta) = p(  \dfrac{sD}{\sqrt{s^2+D^2}}, \beta + \arctan\dfrac{s}{D}).	
\end{array}
\end{equation}

From the experimental point of view, we assume that the origin is positioned at the center of rotation, where the sample is centered. In both geometries, for simplicity and without any loss of generality, a virtual detector will be considered centered at the origin. 
In the standard geometry, the angle $\gamma$ from the central ray indicates a position in the circle centered at the beam source with radius $D$, this circle acts as the detector in this geometry and then it's clearly $2\pi$-periodic on $\gamma$. Whereas in the linear geometry, $s$ is the signed height on the linear detector perpendicular to the central ray and centered at the origin, see Figure \ref{fig:setup}.

The change of variables operation $\MG$ is depicted in Figure \ref{fig:idea}. $\MG$ and $\MS$ are well defined operators because $\beta + \gamma$ and $\beta + \arctan\dfrac{s}{D}$ both belong to $\quotient$.
Since it is true that $\RadonG = \MG\Radon$, a conventional functional relation for space $U$ and $\VG$ provides us with the following statement,
\begin{equation} 
\begin{split} \label{eq:adjG}
    \langle \RadonG f, w \rangle_{\VG} & \ = \  \langle \MG \Radon f, w \rangle_{\VG} \\
    & \ = \ \langle \Radon f, \MGadj w \rangle_{V} \  = \ \langle f, \Radon^* \MGadj w \rangle_{U}. 
\end{split}
\end{equation}

Therefore, it follows that $\RadonGadj \ = \Radon^* \MGadj$. In this work we propose a Fourier approach for equations 
\begin{equation}\label{twosteps}
  \RadonG^* = \Radon^* \MG^* \quad \text{and} \quad  \RadonS^* = \Radon^* \MS^*. 
\end{equation}

\subsection{Integral representation of the adjoint transforms}

Considering both fan-beam geometries, we provide two main results for both the adjoint and backprojection  operators of $\RadonG$ and $\RadonS$. In the following we denote by $\bm{r}_\beta 
$ 
the Cartesian coordinates of the source.

\begin{lemma} \label{theo:fanG}
The operator $\RadonGadj\colon\VG\to U$, defined for $w \in \VG$ and $\bm x \in \R^2$ by
\begin{equation}\label{eq:adjunta-G}
	\RadonGadj w (\bm x)= \int_0^{2\pi}\dfrac{1}{L_\beta} \, w (\gamma_\beta, \beta) \, \dd \beta, 
\end{equation}
is the the adjoint of $\RadonG$ in the sense of (\ref{adjoint}). Here, $L_\beta = \left\| \bm r_\beta - \bm x \right\|_2$ is the Euclidean distance of the source with the backprojected position $\bm x$ and $\gamma_\beta$ is the angle of such point within the fan, i.e., $\cos\gamma_\beta = \dfrac{1}{DL_\beta} \bm r_\beta \cdot (\bm r_\beta - \bm x) $. 
\end{lemma}
\iftrue 
\begin{proof} 
Let $f_\mathsf{P} \in U$ and $w \in \VG$. Using the polar representation (\ref{polar}) of $\Radon$ and the relationship 
\sloppy${\rho \cos(\beta+\gamma-\phi) = L_\beta \sin(\gamma_\beta - \gamma)}$ (see \cite{kak1988}), then
\begin{equation*} 
\begin{split}
	\langle \RadonG f_\mathsf{P} , w\rangle_{\VG} 
	& =  \int_0^{2\pi}\int_0^{\pi} \, w(\gamma, \beta)  \RadonG f_\mathsf{P} (\gamma, \beta) \, \dd \gamma \dd \beta \\
	& =  \int_0^{2\pi}\int_0^{\pi} \, w(\gamma, \beta)  
	\int_0^{2\pi}\int_{\R_+} \, f_\mathsf{P}
	(\rho, \phi) \delta(L_\beta \sin(\gamma_\beta - \gamma)) 
	\, \rho \dd\rho \dd\phi
	\, \dd \gamma \dd \beta,
\end{split}
\end{equation*}
which, from Fubini's theorem, becomes  
\begin{equation*} 
\langle \RadonG f_\mathsf{P} , w\rangle_{\VG} 
= \int_0^{2\pi}\int_{\R_+} f(\rho, \phi)
\int_0^{2\pi}\int_0^{\pi} \,     
  \dfrac{1}{L_\beta} w(\gamma, \beta) \delta(\sin(\gamma_\beta - \gamma))
  \, \dd \gamma \dd \beta
  \, \rho \dd \rho \dd \phi.
\end{equation*}

Therefore, since  $ \delta(\sin(\gamma_\beta - \gamma)) =  \delta(\gamma_\beta - \gamma) $ for $\gamma \in [0,\pi[$ and a fixed $\gamma_\beta$, after performing the above $\gamma$-integration, $\RadonGadj$ defined in (\ref{eq:adjunta-G}) verifies (\ref{adjoint}).
\end{proof}
\fi

%
%
\begin{lemma} \label{theo:fanS}
	The operator $\RadonSadj\colon\VS\to U$, defined for $g \in \VS$ and $\bm x \in \R^2$ by 
	\begin{equation}\label{eq:adjunta-s}
		\RadonSadj g(\bm x) = \frac{1}{D}\int_0^{2\pi}\frac{1}{\ U_\beta} \, \sqrt{s_\beta^2 + D^2} \, g(s_\beta, \beta) \, \dd \beta, 
	\end{equation}
	is the adjoint of $\RadonS$ in the sense of (\ref{adjoint}). $U_\beta$ is the ratio of the scalar projection of $\bm r_\beta - \bm x$ on the central ray to the source-origin distance, and $s_\beta$ is the corresponding height $s$ for $\bm x$ at a source angle $\beta$, i.e., 
	\begin{equation*}
	U_\beta = \dfrac{D^2 - \bm r_\beta \cdot \bm x}{D^2}, \quad s_\beta = \dfrac{\bm x \cdot \bm \xi_\beta}{U_\beta}.
	\end{equation*}

	
\end{lemma}
\begin{proof}
As in the proof of Lemma \ref{theo:fanG}, this follows directly from (\ref{adjoint}) and the polar representation (\ref{polar}).
\end{proof}

The relationship between the adjoint transforms in (\ref{eq:adjunta-G}) and (\ref{eq:adjunta-s}) and the backprojection operators, denoted here by $\BackG$ and $\BackS$ respectively for the standard and linear geometry follows. The backprojection operators are applied e.g., on the FBP algorithm as in the parallel case after the filtering process \cite{kak1988}. Such operators are respectively defined as the weighted integrals 
\begin{equation*}\label{eq:back-G}
	\BackG w (\bm x)= \int_0^{2\pi}\dfrac{1}{L_\beta^2} \, w (\gamma_\beta, \beta) \, \dd \beta, \quad \BackS g(\bm x) = \int_0^{2\pi}\frac{1}{\ U_\beta^2} \, g(s_\beta, \beta) \, \dd \beta, 
\end{equation*}
therefore, the difference with the adjoint transforms lies in the weight factors.

\section{Two-step backprojection formulas}

The two-step adjoint formulations for each geometry, namely $\RadonG^* = \Radon^* \MG^*$ and $\RadonS^* = \Radon^* \MS^*$ derived from (\ref{eq:adjG}) will be detailed here. It is clear that the adjoint of the rebinning operators $\{\MG,\MS\}$ plays an important role in such formulations. 

\begin{lemma}[Adjoint of a rebinning operator]\label{adjoint_rebinning}
Let $\mathcal{M}: X \to Y$ be a bijective and differentiable with continuous inverse rebinning operator between two Hilbert spaces $X$ and $Y$. Its adjoint $\mathcal{M}^* \colon Y \to X$ is then given by $\mathcal{M}^* = \mathsf{J} \, \mathcal{M}^{-1}$, where $\mathsf{J}$ is the Jacobian determinant of the rebinning operation and $\mathcal{M}^{-1}$ its inverse rebinning.
\end{lemma}
\begin{proof}
If $u \in X$ and $v \in Y$, a direct verification of $\langle  \mathcal{M} u , v \rangle_{Y} = \langle  u , \mathcal{M}^*v \rangle_X$ through an integral representation and a change of variables will prove the result. Bijectivity of $ \mathcal{M}$ is needed to preserve the range of integration in both domains of $X$ and $Y$ and for the inverse to exist.
\end{proof}

\begin{lemma}
The adjoint operators of $\MG$ and $\MS$ are respectively defined for $w \in \VG$ and $g \in \VS$ by
\begin{equation}\label{msadj}
	\MGadj w(t, \theta) = w(\gamma(t), \theta - \gamma(t) ) \JG(t)  ,
\end{equation}
\begin{equation}\label{eq.Madj}
    \MS^*g(t, \theta) = g(s(t), \theta-\gamma(t) ) \JS(t), 
\end{equation}
where
\begin{equation}\label{changeof}
\gamma(t) = \arcsin \dfrac{t}{D}, \quad
s(t) = \dfrac{t D}{\sqrt{D^2 - t^2}}, \quad
\JG(t) = \dfrac{1}{\sqrt{D^2 - t^2}}, \quad 
\JS(t) = \dfrac{D^3}{(D^2 - t^2)^{3/2}}.
\end{equation}
\end{lemma}
\begin{proof}
As both operators are bijective \cite{natterer2001mathematical} and continuously differentiable, this is an immediate application of Lemma \ref{adjoint_rebinning}.
\end{proof}

The formal adjoints of $\RadonG$ and $\RadonS$ are presented in Lemmas \ref{theo:fanG} and \ref{theo:fanS}. The problem with 
these formulations is the difficulty for a low-cost implementation algorithm. To circumvent this problem, we use the fact that $\{\Radon, \Radon^*\}$ are bounded operators between Hilbert spaces $U$ and $V$, here considered to be $L^2$ spaces. Hence, we obtain the following result. 

\begin{theo}\label{theotwosteps}
Considering the two fan-beam geometries, the adjoint operators of $\RadonG$ and $\RadonS$, given explicitly by expressions (\ref{eq:adjunta-G}) and (\ref{eq:adjunta-s}), are also given respectively by
$\RadonGadj = \Radon^* \MG^*$ and $\RadonSadj = \Radon^* \MS^*$.
\end{theo}
\begin{proof}
This is an immediate consequence of the uniqueness of the adjoint for bounded operators on Hilbert spaces \cite{rudin1973}. In fact, since $\Radon^*$ and $\MG^*$ are bounded, the composition is also bounded. The same applies for $\MS^*$. 
\end{proof}

To conclude this section, a straightforward two-step rebinning formulation for the backprojection operators is provided here as $\BackG = \Back \MG^{-1}$ and $\BackS = \Back \MS^{-1}$ having an evident justification. The difference mainly lies on the Jacobian determinants $\JG$ and $\JS$ as we can write $\MG^* = \MG^{-1}\JG$ and $\MS^* = \MS^{-1}\JS$. From the adjoint transform expressed as the two-step operation in Theorem \ref{theotwosteps}, we derive our main result in the following section. 

%

\section{Main backprojection theorems}
\subsection{Standard geometry}

The main result of this work is based on a Fourier approach for equations (\ref{twosteps}) and then we make use of the BST formula (\ref{bst}) and the identity $\Radon^* \equiv 2\Back$ for a fixed $\theta$ and $\sigma>0$ in the form 
\begin{equation}\label{idea}
    \widehat{\RadonG^* w}(\sigma\bm \xi_\theta) = \widehat{\Radon^* \MG^* w}(\sigma, \theta) = \dfrac{2}{\sigma}\widehat{\MG^* w}(\sigma, \theta). 
\end{equation}

In \cite[eq.~2.5.6]{katsevich}, a formula relating a parallel sinogram and its backprojection is obtained, similarly to the BST, with a kernel $1/\sigma^{n-1}$ for an arbitrary dimension $n$, also in the frequency domain. This formulation requests however, contrary to the BST, the sinogram $p$ to be the in the range of the Radon operator $\Radon$, i.e., it should exist some $f \in U$ such that $p = \Radon f$. This is not necessarily true in the derivation of  (\ref{idea}), but on the other hand we can safely apply the BST here.

We start announcing our Fourier-based fan-beam adjoint theorem for standard geometry, and then we show how it is easily adapted to the linear case. From (\ref{idea}), we observe the need of computing $\widehat{\MG^*w}$ that is written as a series representation in the following Lemma.

\begin{lemma}\label{serieslema}
Given a standard fan-beam sinogram $w \in \VG$, first define $z \in \VG$ by $z(\gamma,\theta) = w(\gamma,\theta-\gamma)$ for all $(\gamma,\theta) \in \PIquotient \times \quotient$. As $z$ is $2\pi$-periodic with respect to $\gamma$ we can write its Fourier series expansion with the Fourier coefficients 
\begin{equation*}\label{cn}
    c_n(\theta) = \frac{1}{2\pi} \int_0^{2\pi} z(\gamma,\theta) e^{-i n \gamma} \, \dd \gamma,
\end{equation*}
from where we define the coefficients 
\begin{equation}\label{bn}
b_n = 2\pi[c_n + (-1)^n\bar{c_{n}}] \quad \text{for} \ n\geq 1, \quad \text{and} \quad b_0 = 2\pi c_0.
\end{equation}

Then we have a Bessel-Neumann (BN) series description of $\widehat{\MGadj w}$ as
\begin{equation*}\label{eq:BesselSeries}
    \widehat{\MGadj w}(\sigma,\theta) = \sum_{n=0}^\infty b_n(\theta) J_n(D\sigma),
\end{equation*}
where $(J_n)$ is a sequence of Bessel functions of the first kind. 
\end{lemma}

\begin{proof}
Given $w \in \VG$, from (\ref{msadj}) and (\ref{changeof}) the 1D Fourier transform of $\MGadj w(t,\theta)$ with respect to $t$ is
\begin{equation} \label{eq:FadjG}
\begin{split}
\widehat{\MGadj w}(\sigma,\theta) & = 
\int_\R  w( \gamma(t), \theta - \gamma(t) ) e^{-it\sigma} \JG(t) \, \dd t \\
& = \int_0^{2\pi} w(\gamma, \theta - \gamma ) e^{-i D \sigma \sin \gamma}
\, \dd \gamma = \int_0^{2\pi} z(\gamma, \theta) e^{-i D \sigma \sin \gamma}
\, \dd \gamma.
\end{split}
\end{equation}

After expanding $Z$ with $c_n$ we have
\begin{equation*}
\begin{split}
    \widehat{\MGadj w}(\sigma,\theta) & =  \sum_n c_n(\theta) \int_0^{2\pi} e^{i[n\gamma - D\sigma \sin\gamma]} \, \dd \gamma \\
    & = 2\pi \sum_n c_n(\theta) J_n(D\sigma), 
\end{split}
\end{equation*}
from where the result claims using (\ref{bn}) and the property $J_{-n}(x) = (-1)^n J_n(x)$.
\end{proof}

We can now state our fan-beam adjoint theorem for standard geometries.
\begin{theo}[Standard fan-beam adjoint Theorem]\label{teostandard}
The adjoint operator ${\RadonG^* \colon \VG \to U}$ acting on $w \in \VG$ can be expressed in the Fourier domain as the series 
\begin{equation} \label{eq:main}
    \widehat{\RadonG^* w}(\sigma\bm \xi_\theta) = \frac{2}{\sigma} \sum_{n=0}^\infty b_n(\theta) J_n(D\sigma), 
\end{equation}
where $(\sigma,\theta)\in \R_+ \times \quotient$ and the coefficients $b_n$ are computed as in Lemma \ref{serieslema}.
\end{theo}

\begin{proof}
This is an immediate consequence of (\ref{idea}) and Lemma \ref{serieslema}. 
\end{proof}

\begin{remark}\label{remback}
The backprojection operator $\BackG$ can also be written as a BN series. In fact,  from $\BackG = \Back \MG^{-1}$ we have 
\begin{equation*}
\widehat{\BackG w}(\sigma\bm \xi_\theta) = \widehat{\Back \MG^{-1} w}(\sigma, \theta) = \dfrac{1}{\sigma}\widehat{\MG^{-1} w}(\sigma, \theta). 
\end{equation*}

Then, the Jacobian determinant $\JG$ in the derivation (\ref{eq:FadjG}) needs to be compensated. Due to $\dd t = D \cos\gamma \,  \dd \gamma$ in (\ref{eq:FadjG}), we define $\dot z = \mathcal{A}w$ by 
\begin{equation}\label{A}
    \mathcal{A}w(\gamma, \theta) = D\cos\gamma \, w(\gamma, \theta - \gamma),
\end{equation}
and then we compute the Fourier coefficients $(\dot b_n)$ from $\dot z$ in the same way as in (\ref{bn}). Following the same derivation of Theorem \ref{teostandard}, we have the representation
\begin{equation} \label{eq:main_backG}
    \widehat{\BackG w}(\sigma\bm \xi_\theta) = \frac{1}{\sigma} \sum_{n=0}^\infty \dot b_n(\theta) J_n(D\sigma).
\end{equation}

Moreover, from the operation $\beta = \theta - \gamma$ we note that to perform the backprojection with the BN series, we only need a short-scan sinogram where $\beta \in [0, \pi + 2\bar \gamma[$, with $\bar \gamma = \arcsin{\bar x/D}$ the upper bound of $\gamma$ such that the beam intersects the domain of the phantom bounded by $\norm{\bm x}_2 \leq\bar x$. 

\end{remark}
\begin{remark}
The backprojected image obtained with BN series (\ref{eq:main_backG}) assumes a null DC component due to the kernel $1/\sigma$. However, (\ref{eq:main_backG}) can still be used to reconstruct an arbitrary object $f$ with the correct DC component $\kappa$
by estimating it using the volume conservation property of the Radon transform \cite{natterer2001mathematical} that states
\begin{equation*} 
\kappa 
= \int\limits_{\R^2} f(\bm x) \, \dd \bm x 
= \int\limits_\R \Radon f(t, \theta) \, \dd t, \quad \forall \theta.   
\end{equation*}
\end{remark}
Then, using the rebinning operator $\MG$, $\kappa$ can be estimated from fan-beam projections $\RadonG f$ by
\begin{equation*}
\kappa 
= \int\limits_\R \Radon f(t, \theta) \, \dd t 
= \int\limits_\R \MG^{-1}\RadonG f(t, \theta) \, \dd t, \quad \forall \theta.
\end{equation*}

The last integral is computed with a fixed $\theta$ but an average over $\theta$ is also possible giving a more robust $\kappa$ with e.g., noisy projections.

\begin{remark}
Convergence of the BN series (\ref{eq:main_backG}) is guaranteed due to the fact that $J_n(x)$ has a pointwise convergence to $0$ since $|J_n(x)| \leq  |\frac{1}{2}x|^n / n!$ for all $x \in \mathbb R_+$ \cite[eq.~9.1.62]{Abramowitz}.
\end{remark}

\subsection{Linear geometry}

The adjoint operator for the linear fan-beam transform follows. Switching between fan-beam geometries requires only a one-dimensional interpolation on the first variable. In fact, taking $g \in \VS$, the rebinning operator on the first variable $\mathcal L \colon \VS \to \VG$ acting as 
\begin{equation} \label{eq:Ladj0}
\mathcal L g(\gamma, \beta) = 
\left\{\begin{array}{ll}
    g( D\tan \gamma, \beta),& \!\! \text{if} \ \gamma \neq \pi/2\!\!\! \mod \pi\\
    0,& \!\! \text{otherwise}  
\end{array} \right. 
\ \iff \ \RadonG = \mathcal L \RadonS,
\end{equation} 
provides a sinogram $w$ on the space $\VG$. Such bijective operator, according to Lemma \ref{adjoint_rebinning}, has an adjoint operator given by 
\begin{equation} \label{eq:Ladj}
    \mathcal L^* w(s,\beta) =  \mathsf{J}(s) \mathcal L^{-1} w(s,\beta), \quad
\end{equation}
where
\begin{equation*}
    \mathsf{J}(s) = \frac{D}{D^2 + s^2}, \quad
    \mathcal L^{-1} w(s,\beta) = w(\arctan \frac{s}{D},\beta).
\end{equation*}

The following Theorem enables us to provide a adjoint algorithm for the linear fan-beam geometry.

\begin{figure}
\centering
\begin{tikzpicture}
    \filldraw [fill=white, draw=black!60](-1.5,0) circle (1cm);
    \filldraw [fill=white, draw=black!60, dashed] (1.5,0) circle (1cm);
    \filldraw [fill=white, draw=black!60] (5.0,1.2) circle (1cm);
    \filldraw [fill=white, draw=black!60] (5.0,-1.6) circle (1cm);
    \node (U) at (-2.1, 0.4) {$U$};
    \node (V) at ( 2.1, 0.4) {$V$};
    \node (V) at ( 5.7, 1.4) {$\VG$};
    \node (V) at ( 5.7, -1.4) {$\VS$};
    
    \node (b) at (-1.6,-0.0) {${b}$};
    \node (g) at ( 5.0, -1.4) {${g}$};
    \node (w) at ( 5.0, 0.6) {${w}$}; 
    \node (Z) at ( 5.0, 1.8) {${\dot z}$}; 

    \draw[->] (g) [out=-160, in=-40] to node[above]{$\BackS$} (b);
    \draw[->] (Z) [in=40, out=170] to node[above]{$\BackG$} (b);
    \draw[->] (w) -- node[left] {$\mathcal A$} (Z);  
    \draw[->] (g) -- node[left] {$\mathcal L$} (w);  
\end{tikzpicture}
\caption{Diagram for the backprojection operators $\{\BackG, \BackS\}$ (without rebinning to the space $V$) through the action of operators $\{\mathcal L, \mathcal A\}$. See text for details.}
\label{fig:algorithm}
\end{figure}

\begin{theo} \label{theo:fback}
The adjoint operators $\RadonG^*$ and $\RadonS^*$ are related, for all $g \in \VS$ by 
\begin{equation}\label{eq:back_adjs}
    \RadonS^* g = \RadonG^* \tau \mathcal L g, 
\end{equation}
where $\tau(\gamma) = D \sec^2 \gamma$.
\end{theo}
\begin{proof}
From (\ref{eq:Ladj0}) we obtain
$
\RadonGadj = (\mathcal L\RadonS)^* = \RadonS^* \mathcal L^*, 
$
from where $\RadonSadj = \RadonGadj (\mathcal L^*)^{-1}$ follows. Using (\ref{eq:Ladj}) and the fact that 
$(\mathcal L^*)^{-1} = \mathsf{J}( s(\gamma))^{-1} \mathcal L$ with $s(\gamma) = D\tan\gamma$, equation (\ref{eq:back_adjs}) is obtained.
\end{proof}

Therefore, our resulting formulation for the linear fan-beam adjoint transform is obtained by the following construction.
\begin{theo}[Linear fan-beam adjoint Theorem]\label{teolinear}
The operator $\RadonSadj \colon \VG \to U $ acts on $g \in \VS$ as the following two-steps operation. 
\begin{enumerate}
    \item[(i)] Let $\dot w(\gamma,\beta) = \tau(\gamma) w(\gamma,\beta)$ with $w = \mathcal L g$ be in $\VG$. 
    \item[(ii)]  The adjoint operator is obtained as the BN series (\ref{eq:main}) where the coefficients $(b_n)$ in (\ref{bn}) are obtained from $\dot z(\gamma,\theta) = \dot w(\gamma,\theta-\gamma)$.
\end{enumerate}
\end{theo}
\begin{proof}
Each step is justified respectively by Theorems \ref{theo:fback} and \ref{teostandard}.
\end{proof}
\medskip

Finally, the backprojection operator $\BackS$ can also be computed in the same way as in Remark \ref{remback}, by first writing $\BackS g = \BackG \mathcal L g$ and then computing $\BackG$ using (\ref{eq:main_backG}).

Figure \ref{fig:algorithm} illustrates the action of our method, where a backprojection $b\in U$ is obtained in two-steps. The unknown backprojection $\BackS g$ is obtained in the frequency domain from (\ref{eq:main}), using polar coordinates by two main operations to obtain sinogram $\dot z\in \VG$ through $\mathcal L$ defined in (\ref{eq:Ladj0}) (interpolation on the first variable) and $\mathcal A$ defined in (\ref{A}) (interpolating on the second variable). 

\subsection{Mapping to Cartesian coordinates}\label{interp}

A backprojected fan-beam sinogram using Theorems \ref{teostandard} or \ref{teolinear} respectively for standard and linear geometry is obtained in polar coordinates $(\sigma,\theta)\in \R_+ \times \quotient$ in the Fourier domain that needs to be mapped to Cartesian coordinates before applying the inverse Fourier transform. This problem is also present in parallel geometry if using the Fourier reconstruction method based on the FST \cite{natterer2001mathematical} or using the BST approach \cite{miqueles2018backprojection}. The evident strategy is a straightforward (linear) interpolation with the expense of important errors at higher frequencies. 

An error estimate related to this polar-to-Cartesian mapping for the parallel case using the FST is provided in \cite[sec.~V.2]{natterer1986mathematics}. If $\tilde f$ is the reconstruction of $f$ using the Fourier method, the error $\| f - \tilde f\|_{L_2}$ related to the mentioned mapping can be bounded by $c\|f\|_\mathcal{H}$ with some constant $c>0$ and a convenient norm on a Sobolev space $\mathcal{H}$. If a parallel sinogram $g(t,\theta)$ is backprojected within this error bound, then (not conversely, as explained below) it verifies the optimal sampling conditions $\{\Delta t \leq \pi/\Omega, \Delta\theta \leq\pi/(\rho\Omega) \}$ for an object with support included in the disk of radius $\rho$ and $\Omega$ the highest frequency following the Nyquist sampling criterion. Equivalent sampling conditions are also available for fan-beam sinograms in \cite[sec.~IV.3]{natterer2001mathematical}  as $\{\Delta\gamma \leq \pi/(D\Omega), \Delta\beta \leq \pi(D+\rho)/(\rho D\Omega) \}$ in standard geometry and the obvious equivalent in the linear case. It is worth to mention that the highest frequency $\Omega$ is obtained in such references with the Fourier Transform defined with the normalization as in (\ref{cn}).
Conversely, it is also shown in \cite{natterer1986mathematics} that the optimal angular sampling is enough to guarantee the error bound but not the radial sampling in polar frequencies; this can be solved by oversampling by $2$ the Fourier domain with zero-padding. In our case, this is largely verified as the oversampling is performed anyway by a factor at least $2$ until convergence of the BN series. Under these conditions, linear interpolation in the Fourier domain is enough within classical interpolations methods as it will be shown in Figure \ref{fig:mse}\textbf{\textrm{(a)}} where linear interpolation is compared with quadratic splines where results presented negligible differences. In the parallel case, the BST approach is shown to reach competitive results to other $O(N^2\log N)$ methods and overcome them in some scenarios with only linear interpolation from polar to Cartesian coordinates and an adequate zero-padding \cite{miqueles2018backprojection}.

More importantly, advanced polar-to-Cartesian mapping strategies in the context of the Fourier method with FST have shown to compete or even overcome FBP strategies while still presenting the lower $O(N^2\log N)$ complexity. For example, a gridding strategy \cite[sec.~V.2.2]{natterer2001mathematical} can be used by weighting the polar data with an adequate function and then the mapping to Cartesian coordinates is performed by convolution with such weight function. This method handles better interpolations errors than polynomial-based interpolation and reach equivalent results as FBP \cite{marone}. Although no error bounds are presented in such work, the claim is proved by numerical results.  The gridding algorithm is also implemented in the academic reconstruction library \texttt{tomopy} and applied to reconstruct synchrotron tomographic data in different synchrotron facilities \cite{tomopy, marone} instead of FBP. Finally, another strategy is to regularize the passage to Cartesian coordinates using total variation regularization based on the approach developed in \cite{tvct} for parallel tomography with FST, still with an $O(N^2\log N)$ complexity. This technique is motivated by compressed sensing and sparsity priors successfully developed for magnetic resonance imaging (MRI) reconstructions where there is also Fourier data on polar coordinates and it is shown that even with strongly undersampled Fourier data, significantly less polar samples than the optimal conditions given above, high quality reconstructions can be obtained \cite{tao2004}. These two mentioned strategies are available to be used directly on our reconstruction method for fan-beam tomography since the BN series provides polar data on the Fourier domain of the feature object as FST/BST for parallel tomography or MRI does.  

\section{Numerical results}

Simulations are concentrated on the linear case since the standard case is one rebinning operation less than the linear, so results would be slightly better. The Shepp-Logan phantom in Figure \ref{phantom}\textbf{\textrm{(a)}} is employed. We first computed the parallel Radon transform (\ref{radon}) to then obtain a linear fan-beam sinogram through operation $\MS$ defined in (\ref{rebinning}).

\paragraph*{Discrete configuration.}
The domain of the feature phantom is the unit disk $\norm{\bm x}_2 \leq1$ where we have a $N\times N$ uniformly sampled quadrilateral mesh domain. For the domain of $U$ to compute (\ref{radon}), we have $N \times N_\theta$ points uniformly spaced with $(t, \theta) \in [-1,1]\times[0,\pi]$.

For the linear fan-beam geometry, we consider the distance source-origin $D=8$, then $\bar s=D(D^2-1)^{-1/2}$ is the highest detector position that measures the sample on the unit disk subtending an angle $\bar \gamma = \arcsin{1/D}$. We recall that a virtual detector is supposed to be centered, or scaled, at the origin. Due to the symmetry property (\ref{symmetry}) and to Remark \ref{remback}, $(s, \beta)$ needs to cover the sets $[-\bar s, \bar s] \times [0, \pi + 2\bar \gamma[$. For simplicity, we sampled these sets as in the parallel geometry uniformly with $N$ and $N_\theta$ points. Finally, here we will use $N = N_\theta = 512$.

\begin{figure}
    \centering
    \subfloat[]{\includegraphics[width=0.34\textwidth]{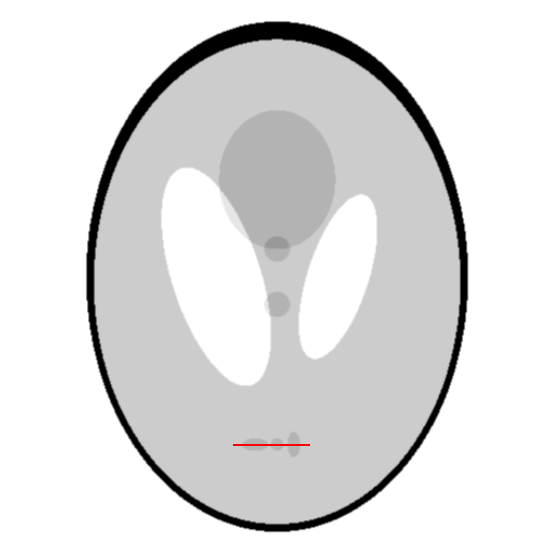}}
    \subfloat[]{\includegraphics[width=0.34\textwidth]{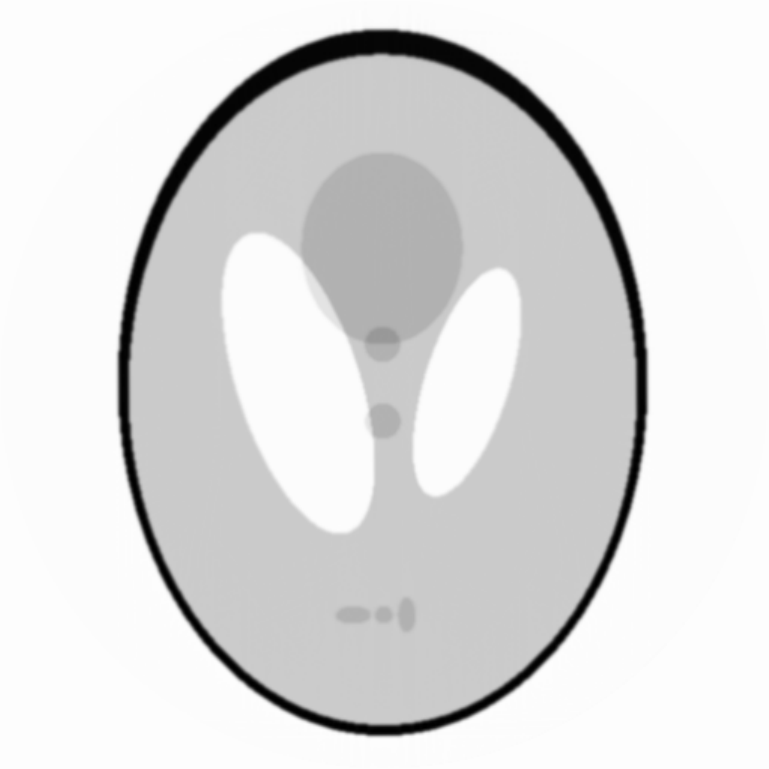}}
    \caption{Feature phantom \textbf{\textrm{(a)}} and reconstruction \textbf{\textrm{(b)}} using the BN series of Theorem \ref{teolinear}.}
    \label{phantom}
\vspace{10pt}
\begin{tabular}{ll}
	\begin{tikzpicture}
	\begin{axis}[xlabel={\small $x$}, ylabel={\small $\BackS (g\star\text{ramp})$}, legend style={at={(0,1)},anchor=south west, font=\footnotesize}, width = .51\textwidth, y label style={at={(axis description cs:.0,.5)},anchor=north}]
	\addplot[color=black] table [x expr=\coordindex/512*2-1, y index=0] {i256.txt}; 
	\addplot[color=blue] table [x expr=\coordindex/512*2-1, y index=0] {d8ln256.txt};
    \legend{phantom, BN series}
	\end{axis}
	\end{tikzpicture} & 
	\begin{tikzpicture}
	\begin{axis}[xlabel={\small $x$}, legend style={at={(0,1)},anchor=south west, font=\footnotesize}, width = .51\textwidth]
	\addplot[color=black] table [x expr=\coordindex/70*0.28-0.16, y index=0] {crop_i.txt}; 
	\addplot[color=blue, mark=*, mark size=1pt] table [x expr=\coordindex/70*0.28-0.16, y index=0] {crop_b.txt};
	\addplot[color=brown, mark=*, mark size=1pt] table [x expr=\coordindex/70*0.28-0.16, y index=0] {crop_brer.txt};
    \legend{phantom, BN series,rebinning}
	\end{axis}
\end{tikzpicture}\\
\textbf{\textrm{(a)}} & \textbf{\textrm{(b)}} \\
\end{tabular}
\caption{\textbf{\textrm{(a)}}: Central profile ($y = 0$) of the feature phantom and the reconstruction through FBP by our BN series method. \textbf{\textrm{(b)}}: Profile over the red segment in Figure \ref{phantom}\textbf{\textrm{(a)}} showing a comparable level of target features reconstructed with both methods.}
\label{fig:profil}
\end{figure}

\begin{figure}
    \captionsetup[subfigure]{labelformat=empty}
    \centering
    \subfloat{\rotatebox{90}{\small BN series\phantom{g}}\ \includegraphics[width=0.32\textwidth]{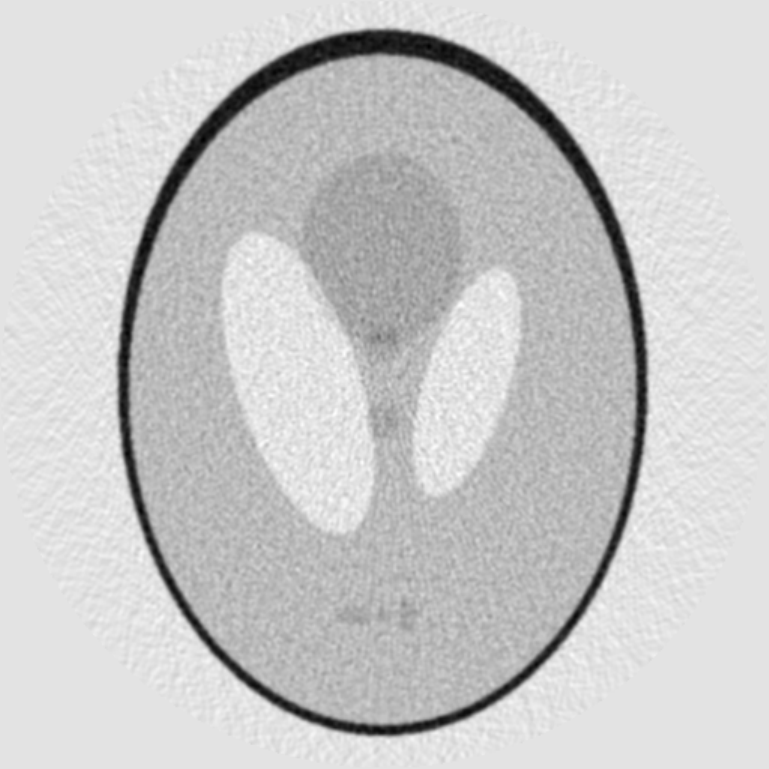}} 
    \subfloat{\includegraphics[width=0.32\textwidth]{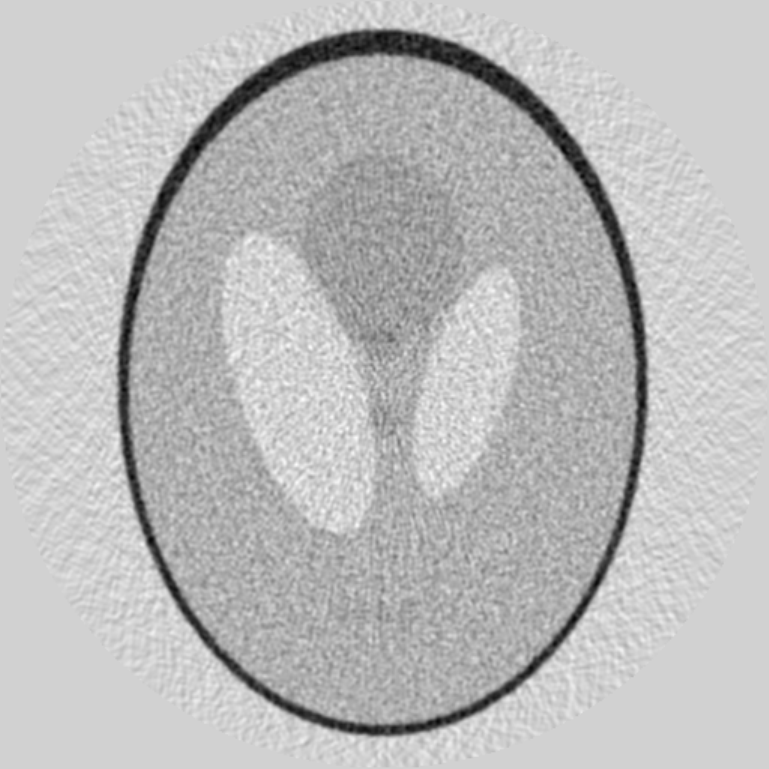}}
    \subfloat{\includegraphics[width=0.32\textwidth]{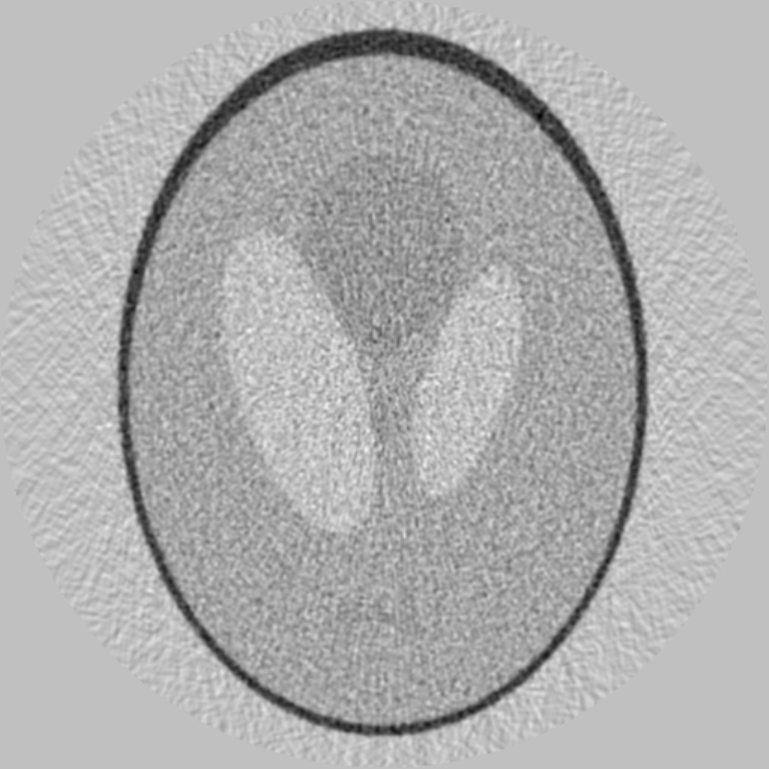}}
    \\\vspace{-10pt}
    \subfloat[\hspace{10pt}Sinogram MSE: $0.5$]{\rotatebox{90}{\small Rebinning}\ \includegraphics[width=0.32\textwidth]{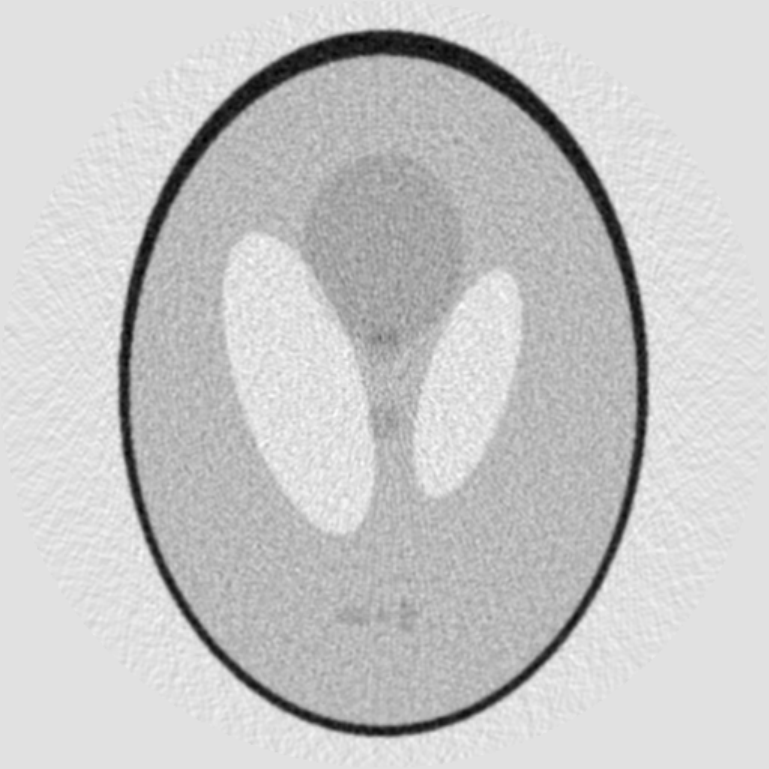}} 
    \subfloat[Sinogram MSE: $2$]{\includegraphics[width=0.32\textwidth]{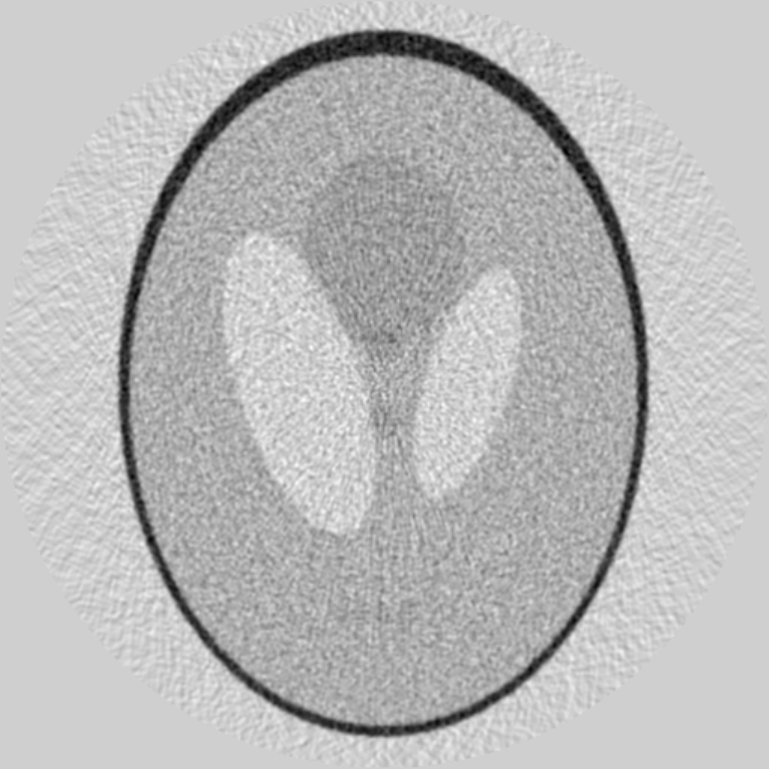}}
    \subfloat[Sinogram MSE:  $4$]{\includegraphics[width=0.32\textwidth]{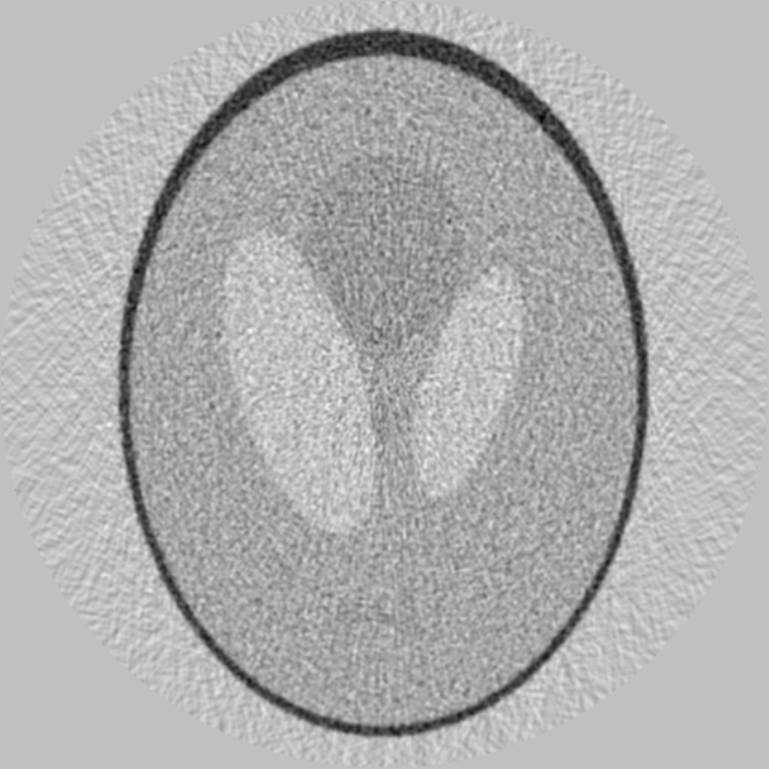}}
    \caption{Reconstructions using the BN series (first row) and rebinning (second row) with increasing sinogram Poisson noise.}
\label{fig:recs}
\vspace{10pt}
\begin{tabular}{ll}
\begin{tikzpicture}
\begin{axis}[xlabel={\small Sinogram MSE}, ylabel={\small Reconstruction MSE}, legend style={at={(0,1)}, anchor=north west, font=\footnotesize}, width = .5\textwidth, height = .4\textwidth, y label style={at={(axis description cs:0.1,.5)}}, xtick={0,1,2,3,4}]
\addplot[color=black, mark=square] coordinates {
	(0, 0.67)
	(0.5, 3.12)
	(1, 5.07)
	(1.5, 6.81)
    (2, 8.59)
    (2.5, 10.51)
    (3, 12.47)
    (3.5, 14.41)
    (4, 16.09)
};
\addplot[color=black, mark=o] coordinates {
	(0, 2.09)
	(0.5, 4.05)
	(1, 5.69)
	(1.5, 7.16)
	(2, 8.63)
	(2.5, 10.23)
	(3, 11.83)
	(3.5, 13.47)
	(4, 14.81)
};
\addplot[color=black, mark=star] coordinates {
	(0, 2.09)
	(0.5, 4.09)
	(1, 5.76)
	(1.5, 7.26)
	(2, 8.77)
	(2.5, 10.42)
	(3, 12.05)
	(3.5, 13.73)
	(4, 15.10)
};
\legend{rebinning, linear BN, quadratic BN}
\end{axis}
\end{tikzpicture} &
\begin{tikzpicture}
\begin{axis}[xlabel={\small Sinogram MSE}, ylabel={\small Reconstruction MAE}, legend style={at={(0,1)}, anchor=north west, font=\footnotesize}, width = .5\textwidth, height = .4\textwidth, y label style={at={(axis description cs:0.1,.5)}}, xtick={0,1,2,3,4}]
\addplot[color=black, mark=square] coordinates {
	(0, 3.41)
	(0.5, 21.92)
	(1, 29.47)
	(1.5, 34.75)
	(2, 39.39)
	(2.5, 43.78)
	(3, 47.78)
	(3.5, 51.54)
	(4, 54.52)
};
	\addplot[color=black, mark=o] coordinates {
	(0, 5.96)
	(0.5, 21.97)
	(1, 28.77)
	(1.5, 33.54)
    (2, 37.72)
    (2.5, 41.68)
    (3, 45.27)
    (3.5, 48.73)
    (4, 51.34)
};
    \legend{rebinning, BN series}
	\end{axis}
\end{tikzpicture} 
\\
\textbf{\textrm{(a)}} & \textbf{\textrm{(b)}} \\
\end{tabular}
\caption{\textbf{\textrm{(a)}}: Reconstruction MSE vs sinogram MSE. 
\textbf{\textrm{(b)}}: Reconstruction MAE vs sinogram MSE.}
\label{fig:mse}
\end{figure}

\paragraph*{Implementation and complexity.}

Implementation details of the two steps of Theorem \ref{teolinear} follow. Step \emph{(i)} is a sequence of two 1D interpolations, $\mathcal{L}$ followed by $\mathcal{A}$ to obtain $\dot z$. $\mathcal{L}$ is a rebinning on the spatial variable from linear to standard geometry, so only needed in the linear case, and $\mathcal{A}$ is a linear operation on the angular variable then not causing strong rebinning errors with linear interpolation. More importantly, both operations are not expensive having complexity $O(N N_\theta)$ and they were easily parallelized as they are both one-dimensional. 

Step \emph{(ii)} is more challenging, we need to compute a truncated BN series (\ref{eq:main_backG}) until convergence. Of course, the number of terms in the sequence $\{J_n\}$ is the same as that of $\{b_n\}$, which itself is the number of samples of $s$ in the sinogram $g$ following the Nyquist sampling theorem \cite{kak1988}. Convergence of the series can thus be reached by first zero-padding the linear variable $s$ of $g$, or equivalently, $\gamma$ of $\dot z$ after performing operation $\mathcal{A}$.   

Computing coefficients $\{b_n\}$ requires $N_\theta$ independent 1D Fourier transforms to obtain (\ref{cn}) followed by the easy operation (\ref{bn}). They were also parallelized in $\theta$ having a complexity $O(N \log N)$ each with Fast Fourier techniques. 
Importantly, the sequence $ \{J_n(D \sigma)\}$ in (\ref{eq:main_backG}) not depending on the sample is computed only once for a given setup and can be used as a \emph{lookup table} for the computation of (\ref{eq:main_backG}), not being costly at computing time. 

After truncating sum (\ref{eq:main_backG}), computing it results simply in a matrix multiplication of the matrices composed with elements $\{\dot b_n(\theta)\}$ and $\{J_n(D \sigma)\}$ for sampled values of $\{\theta, \sigma\}$. The frequency variable $\sigma$ was sampled following the Nyquist sampling theorem. Such matrix multiplication has a complexity of $O(N^3)$ but can be reduced to $O(N^{2.3729})$ with fast matrix multiplication algorithms \cite{matrix, matrix2}, although we used the direct approach in our simulations. We recall that conventional approaches present a cost of $O(N^3)$. Even using a conventional $O(N^3)$ matrix multiplication algorithm, the simplicity and straightforward parallel implementation of matrix multiplication related to a standard backprojection algorithm makes this approach still interesting. Finally, the Fourier domain of the backprojected image is obtained at a polar grid followed by linear interpolation to Cartesian coordinates (see section \ref{interp} for comments on this crucial step) and the 2D inverse (fast) Fourier Transform; operations with respective complexity of $O(N^2)$ and $O(N^2\log N)$. 

\begin{figure}
    \centering
    \begin{tikzpicture}
	\begin{axis}[xlabel={\small Normalized frequency}, ylabel={\small Fourier ring correlation}, legend style={at={(1,1)},anchor=north east, font=\footnotesize}, height = 0.5\textwidth, width = .9\textwidth, y label style={at={(axis description cs:.0,.5)},anchor=north}]
	\addplot[color=black, densely dotted] table [x expr=\coordindex/128, y index=0] {halfbit.txt}; 
	\addplot[color=black] table [x expr=\coordindex/128, y index=0] {phantom_frc.txt};
	\addplot[color=blue] table [x expr=\coordindex/128, y index=0] {BN_series_frc.txt};
	\addplot[color=brown] table [x expr=\coordindex/128, y index=0] {rebinning_frc.txt};
    \addplot[color=black, mark=*, mark size=1pt] coordinates {(0.868704, 0.224887)};	
    \addplot[color=blue, mark=*, mark size=1pt] coordinates {(0.86744, 0.225074)};	
    \addplot[color=brown, mark=*, mark size=1pt] coordinates {(0.877605, 0.224428)};	
    \legend{half-bit threshold, phantom, BN series, rebinning}
	\end{axis}
	\end{tikzpicture} 
    \caption{Fourier ring correlation curves where the obtained resolutions are $(1/0.869,1/0.867,1/0.878)=(1.151, 1.153, 1.139)$ pixels respectively for the phantom, BN series and rebinning methods following the half-bit threshold criterion.}
    \label{fig:frc}
\end{figure}

\paragraph*{Reconstruction results.} 

To obtain image reconstructions, we filtered the parallel sinogram with the conventional ramp filter \cite{kak1988} before applying the rebinning operation $\MS$ to obtain a filtered fan-beam sinogram. We proceeded in this way for sake of simplicity because the filtering processes for fan-beam sinograms, detailed in \cite{kak1988, natterer2001mathematical},  is not in the scope of this work but only the backprojection operation. 

Two backprojections methods were implemented, the conventional rebinning through parallel geometry by operation  $\MS^{-1}$ in (\ref{eq.Madj}) followed by the standard $O(N^3)$ backprojection $\Back$ as presented in theorem \ref{theotwosteps} and our BN series method. The rebinning strategy, used here to evaluate the accuracy of our method, is often employed in available libraries that include fan-beam tomography as in \cite{astra}. The normalized mean square error (MSE), computed as $100 \norm{u- \tilde u}_2^2 / \norm{\tilde u}_2^{2}$, and the normalized mean absolute error (MAE), as $100 \norm{u- \tilde u}_1 / \norm{\tilde u}_1$, for a ground truth image $\tilde u$ and corrupted image $u$ are used as error metrics.

With noiseless data almost perfect reconstructions were obtained with the rebinning method (image not shown) mainly due to fact that the sinogram was obtained through $\MS$. With the BN series the results are also almost exact, shown in Figure \ref{phantom}\textbf{\textrm{(b)}}, the associated central $(y=0)$ profile in Figure \ref{fig:profil}\textbf{\textrm{(a)}} and a particular profile detail in Figure \ref{fig:profil}\textbf{\textrm{(b)}} where both methods show to reconstruct the same level of target features. Spatial resolution is estimated with the Fourier Ring Correlation (FRC) method with the half-bit criterion as suggested in \cite{frc2005}. Three FRC curves are calculated for the phantom and reconstructions with both methods in Figure \ref{fig:frc}. Negligible sub-pixel difference is obtained between the three curves especially in the region of the intersection with the threshold curve. Namely, $(1.151, 1.153, 1.139)$ pixels are respectively obtained as spatial resolution for the phantom, BN series and rebinning methods with a difference of around $0.01$ pixel between phantom and both methods. This shows that no significant resolution is lost with the BN method related to the conventional $O(N^3)$ rebinning method. Reconstruction MSE and MAE are still better with the rebinning method with noiseless data, as indicated in Figure \ref{fig:mse} for the null sinogram MSE case.

To evaluate the noise response of the algorithm, Poisson noise was added to sinograms with increasing sinogram MSE; reconstructed images are shown in Figure \ref{fig:recs} with no significant visual difference between both methods. However, as the noise increases, the BN series method behaves better as shown in Figure \ref{fig:mse}\textbf{\textrm{(a)}} where both reconstruction MSE appears to grow linearly with noise but the BN series with a weaker slope. This is supported with reconstruction MAE behavior in figure \ref{fig:mse}\textbf{\textrm{(b)}}, where the BN series overcomes the rebinning method even earlier (with less noise). 
This could be due to the effect of 2D interpolations of operator  $\MS^{-1}$ on noisy sinograms. The polar-to-Cartesian mapping on the Fourier domain is performed both with linear and quadratic splines interpolation on Figure \ref{fig:mse}\textbf{\textrm{(a)}} as commented in Section \ref{interp} to show that with the adequate sampling, linear interpolation is enough within polynomial interpolation but can be further improved with advanced interpolation as gridding or regularized strategies.

Both algorithms were implemented on a multithreading architecture for each step. We used the Python library \texttt{scipy} \cite{scipy} to compute Fast Fourier transforms, interpolations and Bessel coefficients. Besides the theoretical interest of our formula (\ref{eq:main_backG}), computing time can be reduced with optimized code related to conventional backprojections thanks to the lower computational complexity, when fully exploiting fast matrix multiplication and fast Fourier transform algorithms.

\section{Conclusions}

We have proposed in this work a low-complexity analytic backprojection formulation for fan-beam tomography with standard and linear detectors to be easily implemented as a matrix multiplication algorithm as well as their associated formulations of the adjoint transform on Hilbert spaces. The linear case being more interesting in a synchrotron context was developed as a two-step process. The first step consists in a sequence of simple one-dimensional operators where a standard fan-beam sinogram is obtained. The second step, the backprojection operation, is performed by writing the Fourier transform of the backprojected image as a Bessel-Neumann series on the frequency variable $\sigma$ weighted by $1/\sigma$. The coefficients of the expansion are in fact the Fourier coefficients of the sinogram obtained in the first step. This approach, based on the low cost backprojection BST formula \cite{miqueles2018backprojection} for parallel sinograms, presents a reduction on the computational cost related to conventional fan-beam backprojections when fast matrix multiplication and fast Fourier transforms algorithms are used, and a simpler formulation as a two-matrix multiplication. Iterative reconstruction algorithms where the backprojection process is computed on each iteration, as the EM algorithm \cite{natterer2001mathematical} can also take advantage of this formulation to obtain robust reconstructions.  A second interesting feature is the absence of a strong rebinning process from fan to parallel beam projections as is mostly done in other fan-backprojection algorithms. This results in better handling of noisy data where interpolations can have strong negative effects as shown in our numerical simulations when linearly increasing the noise level and our method outperforms the conventional rebinning strategy. 



\subsection*{Acknowledgements.}
PG was supported by the FWO-SBO MetroFleX project (grant
agreement S004217N).

\bibliographystyle{abbrv} 
\bibliography{main.bib}

\end{document}